\documentclass[a4paper,11pt,reqno]{amsart}
\usepackage{inputenc}
\usepackage{amsmath,amsthm,amsfonts,amssymb}
\usepackage[left=2.20cm,right=2.25cm,top=3cm,bottom=3cm]{geometry}
\usepackage[dvips]{graphicx}
\usepackage[english]{babel}
\usepackage{fouridx}
\usepackage{dsfont}
\usepackage{xcolor}
\usepackage{enumerate}
\usepackage{todonotes}
\usepackage{graphicx, xcolor,fancyhdr}

\usepackage{amsmath,amsthm,amsfonts,amssymb, enumerate, empheq}

\usepackage{mathrsfs, stmaryrd, dsfont}
\usepackage{paralist}

\newcommand\dela[1]{}

\theoremstyle{plain}

\newtheorem{lem}{Lemma}[section]
\newtheorem{thm}[lem]{Theorem}
\newtheorem{prop}[lem]{Proposition}
\newtheorem{cor}[lem]{Corollary}
\theoremstyle{definition}
\newtheorem{definition}[lem]{Definition}
\newtheorem{rem}[lem]{Remark}

\newtheorem{assum}[lem]{Assumption}

\newcommand{\Leb}{\mathrm{Leb}}
\newcommand{\embed}{\hookrightarrow}

\newcommand{\divv}{\mathrm{div}\,}

\newcommand\delb[1]{} 

\newcommand{\eps}{\varepsilon}
\newcommand{\NRJ}[2]{\mathcal{E}[#1](#2)}

\newcommand{\NRJR}[3]{\mathcal{E}_{#1}[#2](#3)}
\newcommand{\ENST}[2]{\mathcal{D}[#1](#2) }

\newcommand{\bv}{{\mathbf{v}}}

\renewcommand{\d}{\mathbf{d}}

\newcommand{\bu}{\mathbf{u}}

\newcommand{\bh}{\mathbf{h}}
\newcommand{\bn}{\mathbf{n}}
\newcommand{\h}{\mathrm{H}}
\newcommand{\sh}{\mathsf{H}}
\newcommand{\MO}{\mathcal{O}}

\renewcommand{\d}{\mathbf{d}}

\newcommand{\bd}{\mathbf{d}}

\newcommand{\el}{\mathbb{L}}
\newcommand{\ve}{\mathrm{V}}
\newcommand{\bue}{\bu^\eps}
\newcommand{\bne}{\bn^\eps}

\newcommand{\CO}{\mathcal{O}}

\DeclareMathOperator{\Div}{div}

\newcommand{\stok}{\mathsf{A}}

\newcommand{\rA}{\mathrm{A}}

\newcommand{\bH}{\mathsf{H}}
\renewcommand{\d}{\mathbf{d}}

\newcommand{\rK}{\mathrm{K}}

\definecolor{lightsg}{HTML}{20B2AA}
\colorlet{TLW}{teal!20!white}

\usepackage[english]{babel}
\usepackage{wrapfig}

\usepackage{times}
\usepackage[T1]{fontenc}

\usepackage{todonotes}

\renewcommand{\rA}{\mathrm{A}}
\newcommand{\rH}{\mathrm{H}}
\newcommand{\rV}{\mathrm{V}}
\numberwithin{equation}{section}

\title[2D stochastic Ericksen-Leslie system ]{On strong solution to the 2D stochastic Ericksen-Leslie system: A Ginzburg-Landau approximation approach }
\author[Z. Brze\'zniak]{Zdzis{\l}aw Brze{\'z}niak}
\address{Department of Mathematics \\
	University of York, Heslington, York YO10
	5DD, UK} \email{zdzislaw.brzezniak@york.ac.uk}

\author[
G. Deugou\'e]{Gabriel Deugou\'e}

\address{Department of Mathematics and Computer Science\\ University of Dschang, Dschang, Cameroon} \email{
	agdeugoue@yahoo.fr}

\author[P. Razafimandimby]{Paul Andr\'e  {Razafimandimby}}
\address{School of Mathematical Sciences \\Dublin City University, Glasnevin, Dublin 9, Ireland} \email{paul.razafimandimby@dcu.ie}
\thanks{Part of this article was written when P. Razafimandimby was a Marie Sk\l{}odowska-Curie  fellow at the University of York.  \\
	\textbf{This article is part of a project that received funding from the  European Union's Horizon 2020 research and innovation programme under the Marie Sk\l{}odowska-Curie grant agreement No. 791735 ``SELEs"}.
}
\begin{document}

\begin{abstract}
	In this manuscript, we consider a highly nonlinear and constrained stochastic PDEs modelling the dynamics of 2-dimensional nematic liquid crystals under random perturbation. This system of SPDEs is also known as the stochastic Ericksen-Leslie equations (SELEs). We discuss the existence of local strong solution to the stochastic Ericksen-Leslie equations. In particular, we study  the convergence the stochastic Ginzburg-Landau approximation of  SELEs, and prove that the SELEs with initial data in $\sh^1\times \sh^2$ has at least a  martingale, local  solution which is strong in PDEs sense. 
\end{abstract}
\maketitle

\section{Introduction}

In 1995, Lin and Liu introduced in \cite{Lin-Liu} the following system of PDEs
\begin{subequations} \label{1a}
	\begin{align}
	&\partial_t \bv+ \bv\cdot \nabla \bv-\Delta \bv +\nabla \mathrm{p}=-\Div~ (\nabla \bd\odot\nabla \bd)  , \text{in $ [0,T)\times O $}\\
	&	\partial_t \bd + \bv\cdot \nabla \bd= \Delta \bd + \lvert \nabla \bd \rvert^2 \bd  ,\\ 
	&	\Div \bv=0,~~\text{in~$[0,T)\times O $},\\
	&	\bv =	{\frac{\partial \bd}{\partial \nu}}=0,\text{ on~$[0,T)\times \partial O $},\\
	&	\lvert \bd \lvert=1,\text{in~$[0,T)\times O $},\label{Eq:Sphere-Constraint}\\
	&	(\bv(0),\bd(0))=(\bv_{0},\bd_{0}),\text{in~$ O$},
	\end{align}
\end{subequations}
as a simplified model describing the motion of a nematic liquid crystal.  The number $T>0$ is a fixed real number, $O$ is a bounded domain with smooth boundary,   the vector fields $\bv:  [0,T)\times O \to \mathbb{R}^2$ and $\bd:  [0,T)\times O\to  \mathbb{R}^3$ represent the velocity and director fields, respectively.  The function $p: [0,T)\times O\to \mathbb{R}$ is the fluids pressure and   $\nabla \bd\odot\nabla \bd$ is the matrix
defined by
\begin{equation*}
[\nabla \bd \odot \nabla \bd ]_{ij}=\sum_{k=1}^3 \partial_i \bd_k \partial_j \bd_k, \quad i,j \in\{1,2\}.
\end{equation*}
For more details on physical modeling of liquid crystal  we refer to the books  \cite{Gennes} and \cite{Stewart} and the papers \cite{Ericksen} and \cite{Leslie}.

The model \eqref{1a} is an oversimplification of a Ericksen-Leslie model of nematic liquid crystal with the one-constant simplification of the Frank-Oseen energy density
\begin{equation*}
\frac12 \lvert  \nabla \bd \rvert^2.
\end{equation*}
However,  the model still  retains many  mathematical and essential features of the hydrodynamic equations for nematic liquid crystals. Moreover,  the mathematical  analysis of the above equations is quite challenging  due to the sphere constraint \eqref{Eq:Sphere-Constraint}, the highly nonlinear coupling term $-\Div~ (\nabla \bd\odot\nabla \bd)$ and  the non-parabolicity of the problem which can be seen from the fact
\begin{equation*}
\Delta \bd + \lvert \nabla \bd \rvert^2 \bd = -\bd\times(\bd\times\Delta \bd), \text{ for } \bd \in \mathbb{S}^2.
\end{equation*}
Because  of these  observations, the system \eqref{1a} has been extensively studied and several important results have been obtained. In addition to the paper \cite{Lin-Liu}  we cited above we refer, among others, to \cite{Hong,Hong12,LLW,
	Lin-Liu2,Lin-Wang,WZZ} for results obtained prior to 2013, and to \cite{YC+SK+YY-2018,Hong14,MCH+YM-2019,Huang14,Huang16,JL+ET+ZX-2016,Lin+Wang16,Wang2014,Wang2016,WZZ15} for results obtained after 2014. For detailed reviews of the literature about the mathematical theory of nematic liquid crystals and other related models, we recommend the review articles \cite{Lin+Wang-2014,MH+JP-2018,Climent} and the recent papers \cite{MCH+YM-2019,Lin+Wang16}.

In this paper, we fix  two numbers $T, \eps>0$ and consider in the 2D torus $\CO$ the following stochastic system
\begin{subequations} \label{Eq-SGL}
	\begin{align}
	&	d\bu+\Bigl[(\bu\cdot\nabla)\bu-\Delta \bu+\nabla p\Bigr]dt=  -\nabla\cdot(\nabla \bn \odot \nabla \bn)dt+ dW,\label{eqn-SLQE-v} \\
	&	\divv \bu=\int_\CO \bu dx = 0,\label{eqn-SLQE-div} \\
	&	d\bn+(\bu\cdot\nabla)\bn dt = \bigl[\Delta
	\bn-\frac{1}{\eps^2}(1-\lvert \bn \rvert^2)\bn] \bigr]dt+(\bn\times \bh)\circ
	d\eta, \label{eqn-SLQE-d}\\
	& \bu(t=0)=\bu_0 \text{ and } \bn(t=0)=\bn_0, \label{Init-Cond}
	\end{align}
\end{subequations}
where $\bu_0: \MO \to \mathbb{R}^d$, $\bn_0:\MO\to \mathbb{R}^3$, $\bh:\CO\to \mathbb{R}^3$ are given mappings; $W$ and $\eta$ are respectively independent cylindrical Wiener process and standard Brownian motion, $\circ d\eta$ is
the Stratonovich integral. 

We should note that the deterministic  version of \eqref{Eq-SGL}, that is,
\begin{subequations}\label{Eq-DGL}
	\begin{align}
	&	\partial_t \bv+(\bv\cdot\nabla)\bv-\Delta \bv+\nabla p=  -\nabla\cdot(\nabla \bd \odot \nabla \bd),\label{eqn-DLQE-v} \\
	&	\divv \bv= 0,\label{eqn-DLQE-div} \\
	&	\partial_t \bd+(\bv\cdot\nabla)\bd  = \Delta
	\bd-\frac{1}{\eps^2}(1-\lvert \bd \rvert^2)\bd, \label{eqn-DLQE-d}\\
	&	\bv=0 \text{ and } \frac{\partial \bd}{\partial \nu}=0 \text{ on }
	\partial O,\label{BC}\\
	& \bv(0)=\bv_0 \text{ and } \bd(0)=\bd_0, \label{Init-Cond-D}
	\end{align}
\end{subequations}
was proposed in \cite{Lin-Liu} as an approximation of the simplified Ericksen-Leslie system \eqref{1a}. 

Our study in this paper is motivated by the need for a sound mathematical analysis of the effect of the stochastic external perturbation on the dynamics of nematic liquid crystals. While the role of noise on the dynamics of the director $\bn$ has been the subject of numerous theoretical and experimental studies in physics,  there are very few rigorous mathematical results related  to models for nematic liquid  under random perturbations. The unpublished manuscript  \cite{BHP13}  is the first paper to prove the existence of strong solution of the stochastic \eqref{Eq-SGL}. This result was generalized in \cite{BHP20} to the case where the quadratic $\mathds{1}_{\lvert \bd \rvert\leq 1} (1-\lvert \bd\rvert^2)\bd $ is replaced by a more general polynomial function.    The paper \cite{BHP18} deals  with weak (both in PDEs and stochastic calculus sense) solutions and the maximum principle. 

Very recently, Hausenblas along with the first and the third author of the present considered in  \cite{BHP19} the stochastic Ericksen-Leslie Equations (SELEs) 
\begin{subequations} \label{SELEs}
	\begin{align}
	&d \bu+[ \bu\cdot \nabla \bu-\Delta \bu +\nabla \mathrm{p}]dt=-\Div~ (\nabla \bn\odot\nabla \bn)dt+ dW  , \text{in $ [0,T)\times \CO $}\\
	&	d \bn + \bu\cdot \nabla \bn dt= [\Delta \bn + \lvert \nabla \bn \rvert^2 \bn ]dt + (\bn\times \bh )\circ d\eta ,\\ 
	&	\Div \bu=0,~~\text{in~$[0,T)\times \CO $},\\
	&	\lvert \bn \lvert=1,\text{in~$[0,T)\times \CO $},\label{Eq:Sphere-Constraint-1}\\
	&	(\bu(0),\bn(0))=(\bu_{0},\bn_{0}),\text{in~$ \CO$}.
	\end{align}
\end{subequations}
Bu using a fixed point method, they showed the existence of local strong solution  $(\bu_0, \bn_0)\in \h^{\alpha}\times \h^{\alpha+1}$ for $\alpha>\frac n2$, where $n=2,3$ is the space dimension.  There is also the paper \cite{Medjo1}  which seeks for a special weak solution $(\bu,\bn)$ of \eqref{SELEs} with the unknown $\bn$ being replaced by an angle $\theta$ such that $\bn=(\cos\theta, \sin\theta)$. This model reduction considerably simplify the mathematical analysis of \eqref{SELEs}.

As we mentioned the model \eqref{Eq-DGL}  was proposed in \cite{Lin-Liu} as an approximation of the simplified Ericksen-Leslie system \eqref{1a}.  It is also  widely used in numerical analysis to handle the sphere constraint \eqref{Eq:Sphere-Constraint}
in the Ericksen-Leslie equations, see for instance  \cite{Walkington}.  Hence,   a natural questions which now arises is  to know whether the solutions $(\bue, \bne )$ converge to a solution to the stochastic Ericksen-Leslie equations as $\eps \to 0$. This question is very interesting and has  been the subject of intensive studies in deterministic case. These  studies have generated several important results which were published in \cite{Hong},  \cite{Hong12} and \cite{Feng+Hong+Mei-2020}. These papers are only related to the convergence of smooth solutions solutions. The convergence of the weak solution remains an open questions. Note that an attempt to solve this open problem was done in \cite{Lin-Liu}, but it is not clear whether the limit satisfies \eqref{1a} or not. 

In the case of the stochastic case, it seems that this note is the first analysis presenting a result on the convergence of the solutions  $(\bue, \bne )$  to \eqref{Eq-SGL}. In particular, we show that by studying the convergence the strong solutions of \eqref{Eq-SGL} w can construct martingale, local strong to \eqref{SELEs} with initial data in $\sh^1\times \sh^2$. Note that  strong solution is taken in the sense of PDEs. The result we obtain is not covered in \cite{BHP19}  which considered  the stochastic ELEs with initial data  $(\bv_0, \bd_0)\in \h^{\alpha}\times \h^{\alpha+1}$ for $\alpha>\frac n2$, where $n=2,3$ is the space dimension.  Moreover, the approaches  are completely different.

Let us now close this introduction by giving the layout of this paper. In Section \ref{Sec:2} we introduce the frequently used notations in this manuscript and our main results, see Theorem \ref{thm-main}. The proof of this main theorem is based on careful derivation of estimates uniform in $\eps$ of the solutions $(\bue, \bne)$ to \eqref{Eq-SGL} and the proof of tightness of laws of these solutions on the space $ C([0,T]; \sh \times \sh^1)\cap C_{\text{weak}}([0,T]; \ve\times \sh^2)\cap L^2_{weak}(0,T; D(A) \times \bH^{3})$ and passage to the limits. These steps are very technical and require long and tedious calculation. Hence, in order to save space we only sketch the main steps of the derivations of uniform estimates in Section \ref{Sec:3}. Also, we only  outline the main ideas of the proof of the tightness and the passage to the limits in Section \ref{Sec:4}.

\section{Notations, the stochastic model and our main result}\label{Sec:2}
\subsection{Notations and the stochastic model}

Let us begin with a brief description of the functional setting.

We will us the symbol $\CO$ to  denote the 2D torus $\mathbb{R}^2/(2\pi \mathbb{Z})^2=\mathbb{R}^2/\sim$, where by $\sim$ we understand  the standard  equivalence relation on $\mathbb{R}^2$ defined by $x=(x_1,x_2)\sim y=(y_1,y_2) $ iff there exist $k\in \mathbb{Z}^2$ such that $y=x+2\pi k$.
It is well known that  $\CO$ can be equipped in a natural differentiable structure so that it becomes a compact riemannian manifold (without boundary). Occasionally it is convenient to view $\CO$ as the square $[0,2\pi]^2$ with the sides identified. In particular, the riemannian volume measure on $\CO$ can be identified with  the Lebesgue measure on $[0,2\pi]^2$ and the riemannian distance is equal to the following one
\[
d([x]_\sim,[y]_\sim)= \sqrt{ \sum_{i=1}^2 \min\{ \vert y_i-x_i\vert, \vert y_i-x_i-2\pi \vert\}^2  },\;\; [x]_\sim,[y]_\sim \in \CO,
\]
where the both representatives  $x=(x_1,x_2)$ and  $y=(y_1,y_2) $ of $[x]_\sim$,  and respectively,  of $[y]_\sim$ have been chosen from $[0,2\pi)^2$.

Throughout, we will  use the following notation 
\[		\mathcal{M}=\{ d: \CO \to \mathbb{R}{^3}: \lvert d(x) \rvert=1 \;\;\mathrm{Leb}\text{-a.e.} \}.
\]

All the vector spaces defined on $\CO$ can also be  defined in terms of functions defined on $[0,2\pi]^2$ satisfying appropriate compatibility conditions on the boundary $\partial ([0,2\pi]^2)=[0,2\pi]\times\{0,2\pi\} \cup \{0,2\pi\} \times [0,2\pi] $. We follow here the presentation from \cite{Temam95} and \cite[Chapert VIII, section 4]{VF88}, see also  recent papers \cite{BD20} and \cite{BBM_2012}. In particular,
we denote by $\mathbb{H}^k(\CO)$, for $k\in\mathbb{N}$,  the Sobolev space of all vector fields defined on $\CO$, equivalently all  $\mathbb{R}^2$ valued functions defined on $[0,2\pi]^2$  satisfying  appropriate compatibility conditions
on the boundary $\partial ([0,2\pi]^2)$,  which are weakly differentiable up to order $k$ and those weak derivatives are square integrable. Obviously, $\mathbb{H}^0(\CO)=\mathbb{L}^2(\CO)$.  We denote by $\mathcal{V}$ the space of all $C^\infty$ vector fields defined on $\CO$, equivalently all  $\mathbb{R}^2$ valued functions defined on $[0,2\pi]^2$  satisfying  appropriate compatibility conditions
on the boundary $\partial ([0,2\pi]^2)$, such that $\mathrm{div}\,u = 0 $.  We also put

\begin{equation}
\label{eqn-L^2_0}
\mathbb{L}^2_0 = \left\{u \in \mathbb{L}^2(\CO) : \int_{\CO} u(x)\,dx = 0 \right\}.
\end{equation}

Then, by $\rH$ we define the closure of the space  $\mathcal{V}$  in the space $\mathbb{L}^2_0$ equipped with the norm and scalar product inherited from the latter space. It is known that $\rH$ is equal to
the set $\left\{ u \in \mathbb{L}^2_0 :\mathrm{div}\,u = 0 \right\}$. We also put

\begin{equation}
\label{eqn-V}
\mathrm{V} = \mathbb{H}^1(\CO) \cap \rH,
\end{equation}
equipped with  the norm and scalar product inherited from the  space $ \mathbb{H}^1(\CO)$. It turns out that  $\rV$ can be equipped with another scalar product and norm defined by

\begin{align}
\label{eqn-norm V}
\langle u, v \rangle_\rV &:= \langle \nabla u, \nabla v \rangle_{L^2}, \;\;u,v \in \rV,\\
\Vert u \Vert_{\rV}^2&:= \langle \nabla u, \nabla u \rangle_{L^2}, u \in \rV.
\end{align}
It it known that the original norm  is equivalent to the new one. We will only use the latter.

We denote by $\mathrm{A}$, the Stokes operator defined by
\begin{equation}
\label{eqn-A}
\begin{split}
\mathrm{D}(\mathrm{A}) &= H^2(\CO) \cap \rH\\
\mathrm{A} \colon \mathrm{D}(\mathrm{A}) &\ni u \mapsto - \Pi \left( \Delta u \right) \in  \rH,
\end{split}
\end{equation}
where
\[ \Pi: \mathbf{L}^2(\CO) \to \rH \] is the orthogonal projection called the Leray--Helmholtz projection.
It is known that $\rA$ is a positive, self-adjoint operator  in ${\mathrm{H}}$ with its inverse $A^{-1}$ being compact.  We will use the following norm on the space $D(\rA)$:
\[|u|^2_{\mathrm{D}(\mathrm{A})} :=  |\mathrm{A} u|^2_{L^2}.\]
Obviously $\mathrm{D}(\mathrm{A})$ is a Hilbert space endowed with that  norm (and the corresponding scalar product).
Moreover, it is known that
\begin{equation}
\label{eqn-A^12}
\mathrm{D}(\rA^{1/2}) = \mathrm{V} \quad \mbox{and} \quad \langle Au , u \rangle_{\mathrm{H}} = \|u\|^2_{\mathrm{V}} = |\nabla u|^2_{L^2}, \,\,\, u \in \mathrm{D}(\rA).
\end{equation}
It is also well known (and follows from   \cite[section 2.2]{Temam95}), that $\Pi$ and $\rA$ commute so that  for every $\theta\geq 0$,
\begin{equation}
\label{eqn-Pi}
\mbox{ $\Pi: \mathrm{D}(\rA^{\theta}) \to \mathrm{D}(\mathrm{A}^{\theta})$ is a bounded linear operator}.
\end{equation}

So far we have introduced mostly the functional spaces corresponding to the velocity field. Let us next introduce the spaces corresponding to the director field.
By $\sh^k$, $k\in \mathbb{N}$, we will denote the Sobolev space of all functions $\bn:\CO \to \mathbb{R}^3$, equivalently all  $\mathbb{R}^3$ valued functions defined on $[0,2\pi]^2$  satisfying  appropriate compatibility conditions  on the boundary $\partial ([0,2\pi]^2)$,  which are weakly differentiable up to order $k$ and those weak derivatives are square integrable. It is well known that $\sh^k$ is a Hilbert space.
Let us recall that by the Sobolev embedding theorem, $\sh^k \embed C(\CO)$ iff $k>1$.

We now give few assumptions and notation about the stochastic perturbations.
\begin{assum}\label{Assum:Usual hypotheses}
	Throughout this paper we are given a complete filtered probability space $(\Omega,
	\mathcal{F}, \mathbb{F}, \mathbb{P})$
	with the filtration $\mathbb{F}=\{\mathcal{F}_t: t\geq 0\}$
	satisfying the usual hypothesis, \textit{i.e.},
	the filtration is right-continuous and all null sets of $\mathcal{F}$ are elements of $\mathcal{F}_0$.
\end{assum}
We introduce what we mean by a cylindrical Wiener process in the following definition.
\begin{definition}\label{def-cylindrical Wiener}
	Assume also that  Assumption \ref{Assum:Usual hypotheses} is satisfied and that $\rK$ is a separable Hilbert space with orthonormal basis $\{e_j: j \in \mathbb{N}\}$. By a $\rK$-cylindrical Wiener process we understand a formal series  
	$ W(t)=\sum_{j=1}^\infty w_j(t) e_j, \;\;\; t\geq 0,  $
	where  $w_j=(w_j(t))_{t\geq 0}$, $j\in \mathbb{N}$,  is a sequence of i.i.d. standard Wiener processes defined on  the filtered  probability space $(\Omega,
	\mathcal{F}, \mathbb{F}, \mathbb{P})$. Equivalently, see \cite[Definition 4.1]{BP_2001-Euler}, 
	a $\rK$-cylindrical Wiener process   defined on  the filtered  probability space $(\Omega,
	\mathcal{F}, \mathbb{F}, \mathbb{P})$  we understand a family $W(t)$, $t\ge 0$ of bounded linear
	operators from $\rK$ into  $L^2(\Omega,{\mathcal
		F},\mathbb{P})$ such that:
	\begin{trivlist}
		\item[(i)] for all $t\ge 0$, and $k_1,k_2
		\in \rK$,
		${\mathbb{E}}\, W(t)k_1 W(t)k_2=t\langle k_1,k_2
		\rangle_\rK$,
		\item[(ii)]
		for each $k\in \rK$, $W(t)k$, $t\ge 0$ is a real valued
		$\mathbb{F}$-Wiener process.
	\end{trivlist}
\end{definition}

Now, by projecting the stochastic model \eqref{Eq-SGL} into the space of divergence free function we obtain the following stochastic PDEs with periodic boundary conditions:
\begin{subequations}\label{GL-LC}
	\begin{align}
	&	d\bu+\biggl[A\bu + \Pi_L (\bu\cdot\nabla\bu)\biggr]dt=-\Pi_L\biggl[\Div (\nabla \bn \odot \nabla \bn)\biggr]dt+ dW \label{GL-LC-1}\\
	&	d \bn+(\bu\cdot\nabla)\bn \,dt=[\Delta \bn-\frac{1}{\eps^2}(1-\lvert \bn \rvert^2)\bn]dt+ (\bn \times \bh )\circ \,d\eta \label{GL-LC-3}\\
	& \bu(t=0)=\bu_0 \text{ and } \bn(t=0)=\bn_0,
	\end{align}
\end{subequations}
where we assume that the initial data satisfies
\begin{equation}\label{eqn-n_0} 
\bn_0\in \mathcal{M},
\end{equation}
and $\circ \,d\eta$ denotes the Stratonovich differential.
\subsection{Our main results}

Let us start with some definitions about stopping times. 
\begin{definition}\label{def-accessible stopping time}
	A random function $\tau:\Omega \to [0,\infty]$ is called a stopping time, see \cite[Definition I.2.1]{Kar-Shr-96}, \cite[Definition 4.1]{Metivier_1982} and \cite[section III.5]{Elw_1982}, iff for each $t\geq 0$, the set
	$\{\omega \in \Omega: t< \tau(\omega)\} \in \mathcal{F}_t$ (or equivalently, $\{\omega \in \Omega: \tau(\omega) \leq t\} \in \mathcal{F}_t$).  A  stopping time
	$\tau:\Omega \to [0,\infty]$ is called accessible, see \cite[section 2.1, p. 45]{Kunita-90},  iff there exists an  increasing
	sequence\footnote{In the sense that for all $n\in \mathbb{N}$, $\tau_n \leq
		\tau_{n+1}$,  $\mathbb{P}$-a.s. } of stopping times  $\tau_n{:\Omega \to [0,\infty)}$ such that $\mathbb{P}$-a.s.
	\begin{inparaenum}
		\item[(i)] for all $n\in \mathbb{N}$,  $\tau_n <
		\tau$;
		\item[(ii)] and $\lim_{n\to \infty} \tau_n =\tau$.
	\end{inparaenum}
	
	The sequence   $(\tau_n)_{n\in \mathbb{N}}$   as above is usually called an announcing sequence for $\tau$.
\end{definition}
We now  continue with the definition of a strong solution to \eqref{Eq-SGL}, see \cite{BHP13} and also \cite{BHP20}. 
\begin{definition}\label{def-strong solution} Assume that $\eps>0$ and  $(\bu_0,\bn_0)\in \ve\times \sh^2$   satisfies the constraint condition \eqref{eqn-n_0}.
	Assume also that  Assumption \ref{Assum:Usual hypotheses} is satisfied.
	A   process  $(\bue,\bne): [0,\infty) \to \ve\times \sh^2$  is called a strong solution to the SGL \eqref{GL-LC} iff
	\begin{trivlist}
		\item[(i)] the process $(\bue,\bne)$ is $\ve$-valued continuous and  $\mathbb{F}$-progressively measurable,
		\item[(ii)] there exits an  $D(A)\times \sh^3$-valued $\mathbb{F}$-progressively measurable process $(\bar{\bue},\bar{\bne})$ such that
		\[
		(\bue,\bne)=(\bar{\bue},\bar{\bne}) \mbox{ almost everywhere w.r.t. } \Leb\otimes \mathbb{P};
		\]
		and, $\mathbb{P}$  almost surely, 
		\begin{align}
		&(\bue, \bne)\in C([0,\infty); \ve\times \sh^2)\mbox{ and } (\bar{\bue},\bar{\bne})\in L^2_{\textrm{loc}}([0,\infty); D(A)\times \sh^3);
		\end{align}
		\item[(iii)] for all  $t \in [0,\infty)$,
		\[\lvert \bne(t) \rvert_{\el^\infty}\leq 1,  \;\; \mathbb{P}\mbox{-almost surely},
		\]
		\item[(iv)] for all  $t \in [0,\infty)$, the  following identities hold true in $\rH$ and $\rH^1$ respectively,  $\mathbb{P}$-almost surely,
		\begin{align}
		\bu(t)&=\bu_0  -\int_0^t \biggl[A\bu + \Pi_L (\bu\cdot\nabla\bu)\biggr]\,ds +\Pi_L\biggl[\Div (\nabla \bn \odot \nabla \bn)\biggr]dt+ W(t), \label{GL-LC-1-int}\\
		\bn(t)&=\bn_0 \int_0^t [-(\bu\cdot\nabla)\bn +\Delta \bn-\frac{1}{\eps^2}(1-\lvert \bn \rvert^2)\bn]\,ds+ (\bn \times \bh )\circ d\eta(s). \label{GL-LC-3-int}
		\end{align}
	\end{trivlist}

\end{definition}
We now recall the following result about the existence and uniqueness of a global strong solution to \eqref{Eq-SGL}, see \cite[Theorem 3.17]{BHP20}. Note that the condition (iii) of Definition \ref{def-strong solution} was proved in \cite[Theorem 5.1]{BHP18}.
\begin{thm} \label{thm-BHR}
 Assume that $\mathbf{h}=h(1,1,1)$  where $h\in \rH^2(\CO,\mathbb{R})$.	Assume  that  Assumption \ref{Assum:Usual hypotheses} is satisfied. 
	Assume that $W=(W(t))_{t\geq 0}$ and  $\eta=(\eta(t))_{t\geq 0}$ are respectively $\ve$ and $\mathbb{R}$ valued Wiener processes defined on the filtered  probability space $(\Omega,
	\mathcal{F}, \mathbb{F}, \mathbb{P})$.  Assume finally that 	  $\eps\in (0,1)$.  Then, for every  $(\bu_0,\bn_0)\in \ve\times \sh^2$ there exists a  process  $(\bue,\bne): [0,\infty) \to \ve\times \sh^2$ 
	which is a unique strong solution to \eqref{GL-LC}.
\end{thm}
A natural questions which now arises is  to know whether the solutions $(\bue, \bne )$ converge to a solution to the stochastic Ericksen-Leslie equations as $\eps \to 0$.  This is the subject of the present paper and it seems that this note is the first analysis presenting a result on the convergence of the solutions   $(\bue, \bne )$  to the SGL. In particular, we obtained the following result. 

\begin{thm}\label{thm-main} Assume that $\rK$ is a separable Hilbert space such that the embedding $\rK \embed \ve$ is Hilbert-Schmidt.  Assume that $\mathbf{h}=h(1,1,1)$  where $h\in \rH^2(\CO,\mathbb{R})$.
	
	There exist a filtered complete probability space $(\Omega, \mathcal{F}, (\mathscr{F}_t)_{t\ge 0}, \mathbb{P})$,  a finite stopping $\tau>0$,  a $\mathrm{K}\times \mathbb{R}$-cylindrical Wiener process $(\tilde{W}, \tilde{\eta})$,  $(\bu,\bn), (\bu^\eps, \bn^\eps): [0,\tau] \to \ve\times \mathsf{H}^2$,
	such that
	$$ (\bue, \bne) \to (\bu,\bn) \text{ a.s. in } C([0,\tau]; D(\stok^{\frac{\alpha-1}2 })\times \mathsf{H}^{\alpha })\cap L^2(0,\tau; D(\stok^\frac{\alpha}{2})\times \bH^{1+\alpha}), \quad \alpha \in [1, 2),$$
	for all $t\in [0,T]$, a.s. $	\bn(t) \in \mathcal{M}$ 
	\begin{align*}
	\bu(t\wedge \tau)-\bu_0= \tilde{W}(t\wedge \tau)-\int_0^{t\wedge \tau}\left(A \bu  +\Pi_L[\bu\cdot \nabla \bu+ \Div(\nabla \bn \odot \bn)]  \right) ds\\
	\bn(t\wedge \tau)-\bn_0=  \int_0^{t\wedge \tau}\left(\Delta \bn  +\lvert \nabla \bn \rvert^2 \bn-\bu\cdot \nabla \bn   \right) ds+\int_0^{t\wedge \tau }(\bn \times \bh)d\tilde{\eta} .
	\end{align*}
\end{thm}
\begin{rem}\label{rem-Hong}
	In the deterministic case, Hong \cite{Hong} proved that
	\begin{equation}
	\frac12 \leq \lvert \bn^\eps(t,x) \rvert_{\mathbb{R}^3}\leq  1, \mbox{ $\forall t\in [0,T)$ a.e. $x$.}
	\end{equation}	
	This is important to handle $\frac{1}{\eps^2}(1-\lvert \bn^\eps\rvert^2)\bn^\eps$ when one wants to study the convergence of the deterministic Ginzburg-Landau approximation. Unfortunately, we do not know how to prove it in the stochastic case. Therefore, we will need to devise an unusual technique.
\end{rem}

The proof of this theorem follows the standard scheme of deriving uniform a priori estimates, establishing tightness in a appropriate spaces, using the famous Jakubowski-Skorokhod representation theorem to pass to the limit. However, the steps of this scheme are quite difficult due to the non-parabolicity of the limiting equations. Moreover, these steps involve long and tedious calculations. Therefore, in order to save space we only give a sketch of the main ideas of the proof of the above theorem.

\begin{rem}\label{rem-improve}
	Note that the previous results obtained in \cite{BHP19} only give the existence of a  local solution $(\bu,\bn): [0,\tau] \to D(A^\frac\alpha2) \times \sh^{1+\alpha}$ whenever $(\bu_0,\bn_0)\in D(A^\frac\alpha2)\times \sh^{1+\alpha}$, $\alpha>\frac d2$, $d=2,3$. Hence, the present note improves the  results  from that paper.
\end{rem}

Throughout, we put
$$ f_\eps(\bn)= \frac1{\eps^2}(1-\lvert \bn \rvert^2)\bn \text{ and } F_\eps(\bn)=\frac{1}{4\eps^2}(1-\lvert \bn \rvert^2)^2.$$

\section{Ideas of the proof of Theorem \ref{thm-main}: Uniform estimates}\label{Sec:3}
\label{sec-proof-part 1}
As mentioned the proof of Theorem consists in deriving uniform estimates, proving tightness results and passage to the limit. In this section we concentrate on the first part, i.e. uniform estimates. In the following
section we will deal with the second and third parts.

In what follows we choose and fix   a separable Hilbert space  $\rK$   such that the embedding $\rK \embed \ve $ is Hilbert-Schmidt. We also assume that Assumptions of Theorem  \ref{thm-BHR}
are satisfied, i.e.  we assume   Assumption \ref{Assum:Usual hypotheses} and that
 $W=(W(t))_{t\geq 0}$ and  $\eta=(\eta(t))_{t\geq 0}$ are respectively $\ve$ and $\mathbb{R}$ valued Wiener processes defined on the filtered  probability space $(\Omega,
		\mathcal{F}, \mathbb{F}, \mathbb{P})$, and   $(\bu_0,\bn_0)\in \ve\times \sh^2$.
We denote by $Q\in \mathscr{L}(\ve)$ the covariance operator of $W$. Here $\mathscr{L}(\ve)$ is the space of all bounded linear maps from $\ve$ into itself.

In this section
we also  fix   	  $\eps\in (0,1]$  denote  by   $(\bue,\bne): [0,\infty) \to \ve\times \sh^2$
the  unique strong solution to the SGL \eqref{GL-LC} guaranteed by Theorem  \ref{thm-BHR}.  Since we want to prove the existence of a local solution, we fix for the remainder of this section a finite time horizon $T>0$.
In all our results below we will find estimates independent of $\eps$. The constants will depend on both the initial data as well as on $T$ but we will not make this dependence explicit.

%
The first estimates we get are given in the following lemma.
\begin{lem}\label{Energy estimates I}
	For any $p\in \mathbb{N}$ there exists a constant $K_0(p)>0$, independent of  $\eps\in (0,1)$, such that
	\begin{align}
	\mathbb{E} \sup_{t\in [0,T)}\left(\lvert \bue(t) \rvert^2_{L^2} + \lvert \nabla \bne(t)\rvert^2_{L^2}+ \lvert F_\eps(\bne(t)) \rvert_{L^1}  \right)^p \leq K_0(p),\\
	\mathbb{E} \left(\int_0^T [\lvert \nabla \bue(t) \rvert^2 + \lvert \Delta \bne(t) + f_\eps(\bne(t)\rvert^2_{L^2} ] dt \right)^p\leq K_0(p).
	\end{align}

\end{lem}
\begin{proof}[Sketch of the proof of Lemma \ref{Energy estimates I}]

The application of the It\^o Lemma \cite{Pardoux_1979} to the functional $ \Gamma_1(\bu)+\Gamma_2(\bn)$, where

		\begin{equation*}
		\Gamma_1(\bu)=\frac12 \lvert \bu \rvert^2_{\el^2}\text{ and }
		\Gamma_2(\bn)= \frac12 \lvert \nabla \bn \rvert^2_{\el^2}+ \frac{1}{4\eps^2}\int_\MO \left[ 1- \lvert \bn \rvert^2\right]^2 dx,
		\end{equation*}
		the use of the fact $f_\eps (\bne) \perp \bne\times \bh$ and
		\begin{equation}
		\langle B(\bue,\bue) + M(\bne), \bue\rangle +\langle \bue\cdot \nabla \bne, f_\eps(\bne) -\Delta \bne\rangle =0,
		\end{equation}
	and the use of the elementary equality	
			$$ \lvert \nabla (\bn^\eps\times \bh ) \rvert^2_{\el^2} +\langle \nabla \bn^\eps, \nabla\left([\bn^\eps \times \bh] \times \bh \right)\rangle=\lvert \bne\times \nabla \bh\rvert^2_{\el^2},$$
		yield the following energy equality which is the basis if the proof of the lemma:
			\begin{equation}\label{Eq:NRJ-IDENT}
			\begin{split}
			& \NRJ{\bu^\eps,\bn^\eps}{t}+2 \int_s^{t }\ENST{\bu^\eps,\bn^\eps}{r} dr=\NRJ{\bu^\eps,\bn^\eps}{s} + \lvert Q \rvert^2_{\mathscr{L}(\ve)}(t-s)\\
			& \qquad + 2   \int_s^{t} \langle   \nabla  \bn^\eps , (\bn^\eps \times \nabla \bh) d\eta  \rangle + 2 \int_s^{t} \langle \bu^\eps, dW \rangle\\
			& \qquad +  \int_s^{t} \lvert (\bn^\eps\times \nabla  \bh ) \rvert^2_{\el^2} dr .
			\end{split}
			\end{equation}
			where
			\begin{align*}
			\NRJ{\bu,\bn}{t}=\int_{\CO} \left[\lvert \bu (t,y)\rvert^2 + \lvert \nabla \bn(t,y) \rvert^2\right] dy+\int_\CO F_\eps(\bn(t,y)) dy,\\
			\ENST{\bu^\eps,\bn^\eps}{t}= \lvert \nabla \bue(t) \rvert^2_{L^2} + \lvert \Delta \bne(t) + f_\eps(\bne(t) ) \rvert^2_{L^2}.
			\end{align*}
			Once we have this energy estimate we can refine the approach in \cite{BHP18} to obtain the estimates in Lemma \ref{Energy estimates I}.
\end{proof}
The above lemma give two natural and important uniform estimates, but they are not sufficient for our purpose. We need to derive uniform estimates in the space
 $C([0,T]; \ve\times \sh^2 ) \cap L^2([0,T]; D(A) \times \sh^3)$.  In order to derive such estimates let us define an important stopping time.

 Let $\delta, R>0$,
$$  \NRJR{R}{\bu,\bn}{t,x}=\int_{B(x,R)} \left[\lvert \bu (t,y)\rvert^2 + \lvert \nabla \bn(t,y) \rvert^2+F_\eps(\bn(t,y)) \right] dy,$$
and define the following three $\mathbb{F}$-stopping times
\begin{align}
\sigma^\eps_1(R)&:=\sigma_1^\eps(\delta,R) = \inf\left\{t\in [0,\infty) : \sup_{x\in \MO} \NRJR{R}{\bue,\bne}{t,x}\ge \delta   \right\} \wedge T,\\
 \sigma_2^\eps &= \inf\left\{ t\in [0,\infty) : \sup_{x\in \MO}\lvert \bne (t,x) \rvert \leq \frac12    \right\}\wedge T,\\
\sigma^\eps(R)&= \sigma_1^\eps(R) \wedge \sigma_2^\eps.
\end{align}

We will now use this stopping time to derive uniform estimates in $C([0,T]; \ve\times \sh^2 ) \cap L^2([0,T]; D(A) \times \sh^3)$ for the stopped processes $(\bu^\eps(\cdot\wedge \sigma^\eps(R)), \bn^\eps(\cdot \wedge \sigma^\eps(R)) )$ fo appropriate choice of $R$. This is motivated by  the theory from the deterministic case which shows that uniform estimates in $C([0,T]; \ve\times \sh^2 ) \cap L^2([0,T]; D(A) \times \sh^3)$  hold if  the energy remains small and $\lvert \bn(t) \rvert_{L^\infty}$ does not enter the ball $B(0,\frac12)$, see \cite[Eq. (3.3)]{Hong}.

Hereafter, we set $\MO_t=[0,t]\times\MO$, $t>0$ and recall the following important lemma, see \cite[Lemma 3.1]{Struwe_1985}. Note that Struwe proved his result on a general compact riemannian manifold and hence his result is valid in our case of a compact 2D torus.

\begin{lem}[The Ladyzhenskaya-Struwe inequality]
	There exist constant $c_0>0$ and $r_1>0$, independent of $\eps\in (0,1]$, such that for every $R\in(0, r_1]$ the following inequality holds
	\begin{equation}
	\begin{split}
	 \lvert\nabla  \bne(t,x) \rvert^4_{L^4(\MO_t)} \leq c_0\left(\sup_{(s,x)\in [0,t] \times \MO } \int_{B(x,R)} \lvert \nabla \bne(s,y)\rvert^2 dy\right) \\
	\times 	\left( \lvert \Delta \bne \rvert^2_{L^2(\MO_t)} + R^{-2} \lvert \nabla \bne \rvert^2_{L^2(\MO_t)} \right).
	\end{split}
	\end{equation}
\end{lem}
	\begin{rem}\label{Rem:Positivity-Approx}
Since $\mathbb{P}$ almost surely the  $(\bue, \bne):[0,T] \to \ve\times \sh^2$ is continuous,   for any $\delta\in (0, 1/8c_0)$  one can find  $r_0>0$ such that for any $R\leq r_0$
		$$\sigma^\eps(R)>0 \text{ a.s. }. $$
	\end{rem}
	
Hereafter, we set
\begin{align}
& R_0=r_1\wedge r_0,\\
\sigma^\eps=\sigma^\eps(R), \;\; \text{ for a fixed } R\in (0, R_0].
\end{align}


	\begin{lem}\label{Energy Estimates IB}
		Let $\delta\in (0, 1/8c_0)$, $p \in \mathbb{N}$ $r_1$ and $r_0$ as  in Remark \ref{Rem:Positivity-Approx}. Let $R_0=r_1\wedge r_0$ and for $R\in (0, R_0]$ we set $\sigma^\eps=\sigma^\eps(R)$.
				 Then, there exists a constant $K_1(p)>0$  independent of $\eps\in (0,1]$, such that
		\begin{equation}\label{Eq-Import-maxreg}
		\mathbb{E}\left(	\int_0^{\sigma^\eps}\lvert \Delta \bne \rvert^2_{\el^2} ds + \frac1{8\eps^4} \int_0^{\sigma^\eps} \lvert 1- \lvert \bne\rvert^2 \rvert^2_{\el^2} ds + \int_0^{\sigma^\eps} \lvert \nabla \lvert \bne\rvert^2 \rvert^2_{\el^2} ds\right)^p \leq K_1(p).
		\end{equation}
	\end{lem}
	\begin{proof}[Sketch of the proof of Lemma \ref{Energy Estimates IB}] 		The idea of the proof consists in the following three steps.
		\begin{itemize}
			\item We expand $\lvert \Delta \bne + f_\eps(\bne)\rvert^2_{\el^2}$ , use integration by parts and the Young inequality to obtain
			\begin{align}
		 \lvert \Delta \bne \rvert^2_{\el^2(\MO_{\sigma^\eps})} + \frac1{\eps^4} \lvert (1-\lvert \bne \rvert^2) \bne \rvert^2_{\el^2(\MO_{\sigma^\eps})} +\frac1{\eps^2} \lvert \nabla \lvert \bne\rvert^2 \rvert^2_{\el^2(\MO_{\sigma^\eps}) }     \nonumber \\
			\le
			\lvert \Delta \bne + f_\eps(\bne) \rvert^2_{\el^2(\MO_{\sigma^\eps})} + 4 \lvert \nabla \bne \rvert^4_{\el^4(\MO_{\sigma^\eps})} + \frac1{8\eps^4} \lvert (1-\lvert \bne\rvert^2) \lvert^2_{L^2(\MO_{\sigma^\eps})}.
			\end{align}
			\item We Use the fact $\lvert \bne(t) \rvert^2 \ge \frac14$ for $t\in [0,\sigma^\eps]$ to control the term containing $\lvert (1-\lvert \bne \rvert^2) \bne \rvert^2$ (this yields the term $\lvert 1- \lvert \bne\rvert^2 \rvert^2$ in the estimates \eqref{Eq-Import-maxreg}! )
			\item We finally use the Ladyzhenskaya-Struwe lemma and Lemma \ref{Energy estimates I} to conclude.
		\end{itemize}
	\end{proof}
Before proceeding further, we recall the following lemma which
was proved in \cite{Hocquet2018} in the case of a general domain. For the case of the torus, it is enough to observe that the $d=2$-dimensional result follows from the $d=1$-dimensional one. In the latter case,
it follows by a simple scaling argument applied to a (large) interval $[0,L]$ with radius $R=1$  with the centers chosen by $x_i=i$, $i=0, \ldots, [L]$. Here $[L]$ denotes the integer part of $L$.   
\begin{lem}\label{Lem:Covering}
	There exists a positive number $C>0$ such that the following holds.
	
	For every $R>0$ there exists a natural number $N_R \in \mathbb{N}$
	such that $N_R  \leq CR^{-2}$ and
	a finite set $\{ x_i : i=1, \cdots, N_R\} \subset \CO$ such that
	\begin{equation}\label{Eq:EqforCover}
	\text{for every $x \in \CO$ there exists $i \in \{1,\cdots,N_R\}$ such that $
		B(x,R) \subset  B(x_i,2R)$ }
	\end{equation}
	Note that in particular $ \CO =\bigcup_{i=1}^{N_R} B(x_i, 2R)$.
\end{lem}
By staying in  $[0,\sigma^\eps]$ and using the above covering lemma we obtain the following lemma.
	\begin{lem}\label{Exponential Estimates}
		Assume that  $\delta\in (0, 1/8c_0)$, $R_0$ as before.  Then, for  every $p \in \mathbb{N}$ there exists a constant $K_2(p)>0$ independent of   $\eps\in(0,1]$
	such that
	\begin{align}
		&	\mathbb{E} \exp{ \left(p \sup_{t\in [0,\sigma^\eps]} \NRJ{\bu^\eps, \bn^\eps}{t}\right)  } \leq K_2(p) \label{Eq:ExpoOfNRJ-0}\\
		&	\mathbb{E} \exp\Bigl(   p \int_0^{\sigma^\eps}\Bigl[\lvert \Delta \bn^\eps \rvert^2_{\el^2} + \lvert \nabla \bu^\eps \rvert^2_{\el^2}\nonumber \\
		& \qquad \qquad \qquad + \frac{1}{8 \eps^4}\lvert [1-\lvert \bn^\eps\rvert^2] \rvert^2_{\el^2} + \frac{1}{\eps^2} \lvert \nabla \lvert \bn^\eps \rvert^2\rvert^2_{\el^2} \Bigr] ds   \Bigr)\leq K_2(p), \label{Eq:ExpoOfEnstrophy-0}\\
		&	\mathbb{E} \exp{\left(p \int_0^{\sigma^\eps}\int_\MO  (\lvert \nabla \bn^\eps \rvert^4+ \lvert \bu^\eps \rvert^4   ) ds  \right)   } \leq K_2(p).\label{Eq:ExpoL4-0}
		\end{align}
	
		
	\end{lem}
\begin{proof}[Sketch proof of Lemma \ref{Exponential Estimates} ]
The first estimate \eqref{Eq:ExpoOfNRJ-0} can be easily obtained. In fact, by covering  the torus $\MO$ by balls $B(x_k,R_0)$ we obtain
	\begin{align*}
	\sup_{t\in [0,\sigma^\eps]} \NRJ{\bu^\eps, \bn^\eps}{t} \leq   \sum_{k=1}^{N_{R_0}}  \sup_{(t,x)\in [0,\sigma^\eps]\times B(x_k, R_0)} \NRJR{R_0}{\bue,\bne}{t,x}
	\leq  N_{R_0}\delta,
	\end{align*}
from which we easily derive \eqref{Eq:ExpoOfNRJ-0}.

The proof of the  second estimate \eqref{Eq:ExpoOfEnstrophy-0} is quite long. We start using the previous estimates and the energy inequality \eqref{Eq:NRJ-IDENT} and derive that
		\begin{align}
		&\lvert \Delta \bne \rvert^2_{\el^2(\MO_{\sigma^\eps})} + \frac1{8\eps^4} \lvert (1-\lvert \bne \rvert^2)\rvert^2_{\el^2(\MO_{\sigma^\eps})} +\frac1{\eps^2} \lvert \nabla \lvert \bne\rvert^2 \rvert^2_{\el^2(\MO_{\sigma^\eps}) }     \nonumber \\
		&\le
		\lvert \Delta \bne + f_\eps(\bne) \rvert^2_{\el^2(\MO_{\sigma^\eps})}\nonumber \\
		& \leq  K_3 + \int_0^{\sigma^\eps} \langle \bue, dW\rangle + \int_0^{\sigma^\eps}\langle \nabla \bne, \bne \times \nabla \bh \rangle d\eta + \frac{K_3}{R_0^2}\int_0^{\sigma^\eps} \lvert \nabla \bne \rvert^2_{\el^2}  ds .\label{Eq:Expo-Entrophy-2}
		\end{align}
Next, by the It\^o formula and the previous  exponential inequality estimate  \eqref{Eq:ExpoOfNRJ-0} we obtain
		\begin{align*} \mathbb{E} \left(\exp\Big[p \int_0^{\sigma^\eps} \langle \bue, dW\rangle + \int_0^{\sigma^\eps}\langle \nabla \bne, \bne \times \nabla \bh \rangle d\eta + \frac{p K_3}{R_0^2}\int_0^{\sigma^\eps} \lvert \nabla \bne \rvert^2_{\el^2}  ds \Big]\right) \leq C,
		\end{align*}
		from which along with \eqref{Eq:Expo-Entrophy-2} we derive \eqref{Eq:ExpoOfEnstrophy-0}.
\end{proof}
We now can derive the following important sets of uniform estimates.

	\begin{lem}\label{Energy Estimates II}
		For every $p\in \mathbb{N}$ there exists a constant $K_3(p)>0$ such that

		\begin{align}
		\mathbb{E}\sup_{ t\in [0,\sigma^\eps]}  [\lvert \nabla \bue(t) \rvert^2_{\el^2}+ \lvert \Delta \bne(t) + f_\eps(\bne(t))\rvert^2_{\el^2}]^p\leq K_3(p), \\
		\mathbb{E} \left[\int_0^{\sigma^\eps} \left(\lvert A \bue(s)  \rvert^2_{\el^2} + \lvert \nabla [\Delta \bn^\eps(s)+ f_\eps(\bn^\eps(s)) ]\rvert^2_{\el^2 } \right)ds \right]^p \leq K_3(p), \\
		\mathbb{E} \left[\int_0^{\sigma^\eps }\left(\left \lvert \frac1\eps(\Delta \bne(s) + f_\eps(\bne(s) ) \cdot \bne(s)  \right\rvert^2_{\el^2}  \right) ds\right]^p\leq K_3(p).
		\end{align}
	\end{lem}
The proof of this lemma is very similar to the proof of the following  key uniform estimates.
\begin{prop}\label{Key uniform estimates}
		Let $\delta\in (0, 1/8c_0)$, $R_0=r_0\wedge r_1$. Then for every $p \in \mathbb{N}$ there exists a constant $K_4(p)>0$ such that
		\begin{align}
		&	\mathbb{E}\left( \sup_{t\in [0,\sigma^\eps]}[\lvert A^\frac12 \bue(t)\rvert^{2 }_{\el^2}+ \lvert \Delta \bne(t) \rvert^{2}_{\el^2}  ]^p\right )\leq K_4(p) ,
\label{Eq:EstRegularityH1H2-A}\\
		&	\mathbb{E}\left[\int_0^{\sigma^\eps} \left( \lvert A\bue \rvert^2_{\el^2}+ \lvert \nabla \Delta \bne \rvert^2_{\el^2}  \right) ds \right]^p\leq K_4(p).	
		\label{Eq:EstRegularityH2H3-A}
		\end{align}
		Moreover,  $\sigma^\eps <T$ is satisfied $\mathbb{P}$ almost surely.
	\end{prop}
To derive the above crucial uniform estimates
 we will need to apply the It\^o formula to the functional $\Lambda:\ve \times \bH^2 \to [0,\infty)$ defined by
\begin{equation}\label{Eq:LambdaDef}
\Lambda(u,d)=\Lambda_1(u) + \Lambda_2(d ),
\end{equation}
where  $\Lambda_1: \bH^2 \to [0,\infty)$ and $\Lambda_2:\ve \to [0,\infty)$ are the energy functionals defined  by
\begin{equation}\label{Eq:Lambda1Lambda2Def}
\Lambda_1(d)=\frac12 \lvert \Delta d  \rvert^2_{\el^2} \text{ and } \Lambda_2(v)= \frac12 \lvert \nabla v \rvert^2_{\el^2}, \;\; v\in \ve, d \in \bH^2.
\end{equation}
We need to establish several lemma involving the first and second Fr\'echet derivatives of $\Lambda_1$ and $\Lambda_2$. Before stating and proving these lemma we recall the formulae for the derivative of $\Lambda_1$
\begin{align}
\Lambda_1^\prime(d)[\mathbf{g}]= \langle \Delta d, \Delta \mathbf{g}\rangle \text{ and }  \Lambda^{\prime \prime}(d)[\mathbf{g}, \mathbf{p}] = \langle \Delta \mathbf{g}, \Delta \mathbf{p} \rangle, \;\; d,\mathbf{g}, \mathbf{p} \in \bH^2.
\end{align}
We state the  following lemma which can be proved using elementary inequalities.
\begin{lem}\label{Lem:Alpha1-Theta1}
	There exists a constant $\alpha_0$ such that for all $v\in D(\stok)$ and $d\in \bH^3$
	\begin{equation}\label{Eq:Alpha1-Theta1}
	\begin{split}
	\Lambda_2^\prime (v )[-v\cdot \nabla v -\Pi_L [\divv(\nabla  d\odot \nabla d)  ]]    = &-  \langle \stok v , \Pi_L [\divv(\nabla d \odot \nabla d)  ] \rangle
	\\  & \leq \frac18 \left(\lvert \nabla \Delta d \rvert^2_{\el^2} + \lvert \stok  v \rvert^2_{\el^2} \right)+ \alpha_0  \lvert \nabla d \rvert^4_{\el^4} \lvert \Delta d \rvert^2_{\el2}.
	\end{split}
	\end{equation}
\end{lem}
\begin{lem}\label{Lem:Alpha2-Theta2}
	There exists a constant $\alpha_1>0$ such that for all $v\in D(\stok)$ and $d\in \bH^2$
	\begin{equation}\label{Eq:Alpha2-Theta2}
	\Lambda_1^\prime (d)[ - v \cdot \nabla d ]=-\langle \Delta d , \Delta(v \cdot \nabla d)      \rangle \leq \frac18 \left(\lvert \nabla \Delta d \rvert^2_{\el^2} + \lvert \stok v \rvert^2_{\el^2} \right)+ \alpha_1 [\lvert \nabla v \rvert^2_{\el^2}+ \lvert \nabla d \rvert^4_{\el^4} ]\lvert \Delta d \rvert^2_{\el^2}.
	\end{equation}
\end{lem}
\begin{lem}\label{Lem:LambdaItoCorrection}
	Let $\bh \in \bH^2$. Then,	there exists a constant $\alpha_3>0$ such that for all $d\in \bH^2$
	\begin{equation}\label{Eq:LambdaItoCorrection}
	\frac12  \Lambda^\prime_1(d)[(d\times\bh)\times \bh] + \frac12 \Lambda_1^{\prime\prime}(d)[d\times \h] \leq \alpha_3 \lvert \bh \rvert^2_{\bH^2}\left[ \lvert \Delta d \rvert^2 + \lvert \nabla d \rvert^2_{\el^4}+ \lvert d \rvert^2_{\el^\infty} \right].
	\end{equation}
\end{lem}
\begin{lem}\label{Lem:LambdaBDGNeeded}
	Let $\bh \in \bH^2$.  Then, there exists a constant $\alpha_4>0$ such that for all  $d\in \bH^2$
	\begin{equation}\label{Eq:LambdaBDGNeeded}
	\lvert \Lambda_1^\prime(d)[d\times \bh] \rvert^2\leq \alpha_4 \lvert \bh \rvert_{\sh^2}^2 \lvert\Delta d \rvert_{\el^2}^2  \left(    \lvert \Delta d \rvert_{\el^2}^2 + \lvert \nabla d \rvert_{\el^4}^2+ \lvert d\rvert_{\el^\infty}^2 \right).
	\end{equation}
\end{lem}
One of the most difficult term to control in the application of It\^o formula for $\Lambda(v,d)$ is the term involving the Ginzburg-Landau functional $f_\eps(d)$. However, with skillful and careful analysis we were able to derive the following important result.
\begin{lem}\label{Lem:Alpha3}

	There exists a constant $\alpha_2>0$ such that for all $\eps\in (0,1]$ and 	$d\in \sh^3$ be satisfying
	\begin{equation}\label{Eq:BoundednessDir}
	\frac12< \lvert d(x) \rvert^2 \leq 1 \text{ for all } x\in \CO,
	\end{equation}
	\begin{equation}\label{Eq:Alpha3}
	\begin{split}
	\Lambda_1^\prime(d)[f_\eps(d)]=	\langle \Delta d, \Delta f_\eps(d) \rangle \leq & \alpha_2[\lvert \nabla (\Delta d +f_\eps(d) )\rvert^2_{\el^2} + \lvert  \Delta d +f_\eps(d)  \rvert^4_{\el^2} +\lvert \nabla d \rvert^4_{\el^4} +
	\lvert \Delta d \rvert^2_{\el^2} \lvert \nabla d \rvert^2_{\el^4}      ]\\
	&\qquad + \frac14 \lvert \nabla \Delta d \vert^2_{\el^2} -\frac{1}{2\eps^2} \lvert \Delta(1- \lvert d\rvert^2) \rvert^2_{\el^2}\\
	& \qquad -\frac1{\eps^2} \int_\CO (1-\lvert d \rvert^2) [\lvert \nabla^2 d \rvert^2 + \lvert \Delta d \rvert^2 ] dx
	\end{split}
	\end{equation}
	
\end{lem}
\begin{proof}
Let	$\eps\in (0,1]$ and 	$d\in \sh^3$ satisfying
	\begin{equation}\label{Eq:BoundednessDir-1}
	\frac12< \lvert d(x) \rvert^2 \leq 1 \text{ for all } x\in \CO.
	\end{equation}
Now, in order to prove the lemma we will need the following identity which is taken from \cite{Hong14} 
	\begin{equation}
	d \cdot \Delta^2 d = \frac12 \Delta^2\lvert d\rvert^2 -4 \nabla d \nabla \Delta d - 2 \lvert \nabla^2 d \rvert^2 - \lvert \Delta d\rvert^2.
	\end{equation}
	We also need the following inequality which follows from \eqref{Eq:BoundednessDir}
	\begin{equation}\label{Eq:DerFeps}
	\begin{split}
	\left\lvert\frac1{\eps^2} (1-\lvert d\rvert^2 )  \right\rvert = \left\lvert\frac1{\eps^2} (1-\lvert d\rvert^2 ) d \right\rvert ( \lvert d \rvert)^{-1}\leq 2 \lvert f_\eps(d) \rvert.
	\end{split}
	\end{equation}
	With these two observations in mind we have
	\begin{align}
	\lvert (\Delta d, \Delta f_\eps(d) )  =&\frac1{\eps^2} (\Delta^2 d, (1-\lvert d\rvert^2) d   )\\
	=& \frac1{\eps^2} ((1-\lvert d\rvert^2), \frac12 \Delta^2 \lvert d\rvert^2 - 4 \nabla d\nabla \Delta d -2 \lvert \nabla^2 d\rvert^2 -\lvert \Delta d \rvert^2     )\\
	= & -\frac1{\eps^2} \lvert \Delta (1-\lvert d\rvert^2) \rvert^2_{\el^2} -\frac1{\eps^2} \int_\CO (1-\lvert d \rvert^2) [\lvert \nabla^2 d \rvert^2 + \lvert \Delta d \rvert^2 ] dx\\
	& +\qquad \frac{4}{\eps^2} \int_\CO [(1- \lvert d\rvert^2) \nabla d\nabla \Delta d ]dx \\
	&=I_1 + I_2+ I_3. \label{Eq:Alpha3-StepOne}
	\end{align}	
	It is clear that $I_1\ge 0$, hence we do not need to worry about it. Since $(1-\lvert d\rvert^2) \ge 0$ we do not need to deal with $I_2$.
	Let us then estimate $I_3$. For doing so we use \eqref{Eq:DerFeps}  and get
	
	\begin{equation}
	\begin{split}
	I_3 \leq &8 \int_\CO [\lvert (1-\lvert d\rvert^2) \lvert \nabla d\rvert \lvert \nabla \Delta d\rvert ]dx\\
	\leq & 8 \int_\CO [\lvert f_\eps(d) +\Delta d -\Delta d \rvert \lvert \nabla d\rvert \lvert \nabla \Delta d\rvert ] dx\\
	\leq & 8 \int_\CO[ \lvert \Delta d +f_\eps(d) \rvert \lvert \nabla d\rvert \lvert \nabla \Delta d\rvert] dx + \int_\CO [\lvert \Delta d \rvert \lvert \nabla d\rvert \lvert \nabla \Delta d\rvert] dx\\
	=& J_1 + J_2. \label{Eq:Alpha3-Step2}
	\end{split}
	\end{equation}
	Using the Young, the H\"older, the Gagliardo-Nirenberg inequalities and the Young inequality in this order yields that for any $\alpha>0$ there exists a constant $C(\alpha)>0$ such that
	\begin{align}
	J_1  \le& \alpha \lvert \nabla \Delta d \rvert^2_{\el^2} + C(\alpha ) \int_\CO (\lvert \Delta d+ f_\eps(d)\rvert^2 \lvert \nabla d \rvert^2  ) dx\\
	\leq &  \alpha \lvert \nabla \Delta d \rvert^2_{\el^2} + C(\alpha  \lvert \Delta d + f_\eps(d) \rvert^2_{\el^4} \lvert \nabla d \rvert^2_{\el^4}\\
	\leq & \alpha \lvert \nabla \Delta d \rvert^2_{\el^2} + C(\alpha) \lvert \Delta d +f_\eps(d) \rvert_{\el^2} \lvert \nabla(\Delta d +f_\eps(d)) \rvert_{\el^2} \lvert \nabla d \rvert^2_{\el^4}\\
	\le& \alpha \lvert \nabla \Delta d \rvert^2_{\el^2} + C(\alpha)\Big [ \lvert \nabla(\Delta d +f_\eps(d)) \rvert_{\el^2}^2+\lvert \Delta d +f_\eps(d) \rvert_{\el^2}^2  \lvert \nabla d \rvert^4_{\el^4}  \Big ] \label{Eq:Alpha3-J2}
	\end{align}
	
	Next, we deal with $J_2$ in a similar way.  Using the Young, the H\"older, the Gagliardo-Nirenberg inequalities and the Young inequality in this order yields that for any $\alpha>0$ there exists a constant $C(\alpha)>0$ such that
	\begin{align}
	J_2 \le& \alpha \lvert \nabla \Delta d \rvert^2_{\el^2} + C(\alpha ) \int_\CO ( \lvert \Delta d\rvert^2 \lvert \nabla d \rvert^2  ) dx\\
	\leq &  \alpha \lvert \nabla \Delta d \rvert^2_{\el^2} + C(\alpha) \lvert \Delta d \rvert^2_{\el^4} \lvert \nabla d \rvert^2_{\el^4}\\
	\leq & \alpha \lvert \nabla \Delta d \rvert^2_{\el^2} + C(\alpha) \lvert \Delta d  \rvert_{\el^2} \lvert \nabla\Delta d  \rvert_{\el^2} \lvert \nabla d \rvert^2_{\el^4}\\
	\le& \alpha \lvert \nabla \Delta d \rvert^2_{\el^2} + C(\alpha)\Big [\lvert \Delta d \rvert_{\el^2}^2  \lvert \nabla d \rvert^4_{\el^4}  \Big ]\label{Eq:Alpha3-J1}
	\end{align}
	The inequality \eqref{Eq:Alpha3} follows from \eqref{Eq:Alpha3-StepOne}, \eqref{Eq:Alpha3-Step2} , \eqref{Eq:Alpha3-J1} and \eqref{Eq:Alpha3-J2} by choosing $\alpha=\frac14$.
\end{proof}

Let us sum up our findings from the above lemma in the next remark.
\begin{rem}
	Let $\bh \in \bH^2$, $v\in D(\stok)$ and $d\in \sh^3$ be satisfying
	\begin{equation}\label{Eq:BoundednessDir-2}
	\frac12< \lvert d(x) \rvert^2 \leq 1 \text{ for all } x\in \CO.
	\end{equation}
	Let $\alpha_0, \alpha_1, \alpha_2$ be the constants from Lemmata \ref{Lem:Alpha1-Theta1}-\ref{Lem:Alpha3} and let us put
	\begin{align}
	& \mathsf{R}_1(d):= \alpha_2\left[\lvert \nabla (\Delta d +f_\eps(d)) \rvert^2_{\el^2} +\lvert \Delta d +f_\eps(d) \rvert^4_{\el^2} + \lvert \nabla d \rvert^4_{\el^4} \right],\\
	& \mathsf{R}_2(d):= \alpha_3 \left[\lvert \Delta d \rvert^2_{\el^2} + \lvert \nabla d \rvert^2_{\el^4} + \lvert d \rvert^2_{\el^\infty}   \right],\\
	&\mathsf{S}(v,d):= \lvert \rA v \rvert^2_{\el^2}+ \lvert \nabla \Delta d \rvert^2_{\el^2}+\left \lvert \frac1{\eps}\Delta (1-\lvert d \rvert^2) \right\rvert^2_{\el^2}+\left \lvert \frac1\eps\sqrt{(1-\lvert d\rvert^2)} \nabla^2 d\right\rvert^2_{\el^2}+\left \lvert \frac1\eps\sqrt{(1-\lvert d\rvert^2)} \Delta d\right\rvert^2_{\el^2},\\
	& \mathsf{N}_1(d):=[\alpha_0 +\alpha_1+\alpha_2] \lvert \nabla d \rvert^4_{\el^4}.
	\end{align}
	Then, it follows from Lemma \ref{Lem:Alpha1-Theta1}-\ref{Lem:LambdaItoCorrection} that
	\begin{equation}\label{Eq:WildTerm}
	\begin{split}
	&	\Lambda_2^\prime(v)[-\rA d -v\cdot \nabla v-\Pi_L(\Div[\nabla d \odot \nabla d] ) ] +
	\Lambda_1^\prime(d)[\Delta d+f_\eps(d)-v\cdot \nabla d+\frac12(d\times\bh)\times \bh] + \frac12\Lambda^{\prime\prime}[\Delta (d\times \bh)]\\
	&\qquad 	\leq  -\mathsf{S}(v,d) + \mathsf{R}_1(d)+ \lvert \bh \rvert^2_{\bH^2} \mathsf{R}_2(d)+\mathsf{N}_1(d) \Lambda(v,d).
	\end{split}
	\end{equation}
\end{rem}

Bearing the notation of this remark in mind, we set
\begin{equation*}
\Phi_1 (t\wedge \sigma^\eps)= \exp\left({-\int_0^{t\wedge \sigma^\eps}  \mathsf{N}_1(\bne(s)) ds   }\right), \;\; t\ge 0.
\end{equation*}

We now state and sketch the proof of  the following result.
\begin{prop}\label{Prop:EstRegularityH1H2H3}
	For any $p\in \mathbb{N}$ there exist constants $K_5(p), \mathfrak{C}_6(p)>0$, independent of $\eps>0$, such that
	\begin{align}
		\mathbb{E}\bigl(& \sup_{t\in [0,T]}[\lvert \stok^\frac12 \bv(t\wedge \sigma^\eps)\rvert^{2 }_{\el^2}+ \lvert \Delta \bn(t\wedge \sigma^\eps) \rvert^{2}_{\el^2}  ]^p\bigr)\leq K_5(p),\label{Eq:EstRegularityH1H2}\\
		\mathbb{E}\Bigl[ &\int_0^{T\wedge \sigma^\eps} \bigl( \lvert \stok\bv \rvert^2_{\el^2}+ \lvert \nabla \Delta \bn \rvert^2_{\el^2}  +\frac{1}{2\eps^2} \lvert \Delta(1- \lvert d\rvert^2) \rvert^2_{\el^2}
\nonumber\\
 &+\frac1{\eps^2}  \bigl\lvert\sqrt{(1-\lvert d \rvert^2)} \nabla^2d \bigr\rvert^2_{\el^2} + \bigl\lvert \sqrt{(1-\lvert d \rvert^2)} \Delta d \bigr\rvert^2_{\el^2} \bigr) ds \Bigr]^p \leq K_6(p).	
		\label{Eq:EstRegularityH2H3}
	\end{align}
\end{prop}
\begin{proof}
	The proof  involves long and tedious calculation, so  we will only outline the main idea. Without of loss generality we only prove the estimate for $p\in \mathbb{N}$.
	
	We need to use the It\^o's formula for several processes.
	We firstly apply It\^o's formula to $\Lambda_2(\bue(t\wedge\sigma^\eps))$ and $\Lambda_1(\bne (t\wedge \sigma^\eps))$, then to $Z(t\wedge \sigma^\eps)$ where
$$Z(t\wedge \sigma^\eps)= \Phi_1(t\wedge \sigma^\eps) \Lambda(\bue(t\wedge \sigma^\eps), \bne(t\wedge \sigma^\eps)),\quad t\in [0,T],$$ and  $\Lambda$ is defined in \eqref{Eq:LambdaDef}. Using  \eqref{Eq:WildTerm}, the uniform estimates in Lemmata  \ref{Energy estimates I}, \ref{Energy Estimates IB}, \ref{Energy Estimates II} and Proposition \ref{Key uniform estimates}, and the fact $\Phi_1 \leq 1$  we infer that for all $p\ge 1$ there exist constants  $K_7>0$ (depending only on $p$) and $ K_8>0$ which depends only on  $p$, $T$, $\lvert Q \rvert_{\mathscr{L}(\ve)}$ ,$ \lvert \bh \rvert^2_{\bH^2}$ and $\lvert(\bv_0,\bn_0)\rvert_{\ve\times \bH^2}^{4p}$ such that for all $\eps\in (0,1]$
	\begin{align*}
	\mathbb{E} \sup_{s\in [0,t]} [Z(s\wedge \sigma^\eps)]^p &+ \mathbb{E} \left[\int_0^{t\wedge \sigma^\eps}\Phi_1(s) \mathsf{S}(\bue(s), \bne(s)) ds \right]^p \leq  K_7 \bigl( \mathbb{E} [Z(0)]^p + \lvert \bh \rvert_{\bH^2}^{2p}\left(\mathbb{E} \int_0^{T\wedge \sigma^\eps} \mathsf{R}_2(\d(s)) ds \right)^p \\
	&+\mathbb{E} \left(\int_0^{t\wedge \sigma^\eps} \mathsf{R}_1(\d(s))ds \right)^p  + \mathbb{E} \sup_{s\in [0,t]} \lvert M(s\wedge \sigma^\eps)\rvert^p   \bigr)\\
	&\leq  K_8+ K_7 \mathbb{E}  \sup_{s\in [0,t]} \lvert M(s\wedge \sigma^\eps)\rvert^p,
	\end{align*}
	where the process $M$ is defined by 
	\begin{align*}
	M(s)=\int_0^{s} \Phi_1(r)\Lambda_2^\prime(\bue(r))  dW(r) + \sum_{j=1}^\infty \int_0^{t} \Phi_1(r) \Lambda_1^\prime(\bne(r)) [\bne(r) \times \bh] \circ d\eta(r), \;\; s\in [0,T].
	\end{align*}
	We now use the Burkholder-Davis-Gundy,  the H\"older, the Young inequalities, the uniform  estimates in Lemmata \ref{Energy estimates I}, \ref{Energy Estimates II} and Proposition \ref{Key uniform estimates}, and the fact $\Phi_1 \leq 1$   to deduce that
	\begin{equation}
K_7 \mathbb{E}  \sup_{s\in [0,t]} \lvert M(s\wedge \sigma^\eps)\rvert^p \leq \frac12 	\mathbb{E} \sup_{s\in [0,t]} [Z(s\wedge \sigma^\eps)]^p  + K_8.
	\end{equation}
	Collecting all the above estimates yield there exists a constant $ K_9>0$ which depends only on  $p$, $T$, $\lvert \bh \rvert^{2p}_{\bH^2}$ and $\lvert(v_0,\bn_0)\rvert_{\ve\times \bH^2}^{4p}$ such that for all $\eps\in (0,1]$
	\begin{equation}\label{Eq:ResutofGronwalforUpsilon}
	\begin{split}
	\mathbb{E} \sup_{s\in [0,t]} [Z(s\wedge \sigma^\eps)]^p + \mathbb{E} \left[\int_0^{t\wedge \sigma^\eps}\Phi_1(s) \mathsf{S}(\bue(s), \bne(s)) ds \right]^p \leq  K_9.
	\end{split}
	\end{equation}
 With this at hand, we can now estimate the $\mathbb{E} \sup_{s\in [0,T]} [\Psi(\bue,\bne)(s\wedge \xi_\eps)]^p$ as follows. Let $$\Phi_1^{-1}= 1/\Phi_1=\exp\left(\int_0^\cdot \mathsf{N}_1(\bne(s)) ds   \right).$$ Since $\Phi_1^{-1}$ is an increasing function of the time $t$ we then obtain
\begin{equation}
\mathbb{E} \sup_{s\in [0,T]}  [\Psi(\bue,\bne)(s\wedge \xi_\eps)]^p= \mathbb{E} \sup_{t\in [0,T] } \left[\Phi_1^{-1}(s\wedge \xi_\eps) Z(s\wedge \xi_\eps)\right]^p\le\left( \mathbb{E}\lvert \Phi_1^{-1}(T\wedge \xi_\eps) \rvert^{2p} \mathbb{E} \sup_{s\in [0,T]} Z^{2p}(s)\right)^\frac12
\end{equation}
from  which along with the definition of $\Phi_1$ and the exponential estimates in Lemma \ref{Exponential Estimates} and \eqref{Eq:ResutofGronwalforUpsilon} we derive that for any $R>0$, $p\ge 1$ there exists a constant $K_{10}$ such that for all $\eps\in [0,1)$
\begin{equation}\label{Eq:FinalStepofEstimateOfPsiVD}
\mathbb{E} \sup_{s\in [0,T]} \left[ \lvert \nabla \bue(s\wedge \xi_\eps ) \rvert^2_{\el^2}  + \lvert \Delta \bne(s\wedge \xi_\eps) \rvert^2_{\el^2}    \right]^p \leq K_{10} .
\end{equation}
We establish the estimate \eqref{Eq:EstRegularityH2H3} in a similar way. This completes the proof of the proposition.
\end{proof}

\section{Ideas of the proof of Theorem \ref{thm-main}: Tightness and passage to the limit}\label{Sec:4}
\label{sec-proof part 2}

As in the previous section, in what follows we choose and fix   a separable Hilbert space  $\rK$ such that the embedding $\rK \embed \ve $ is Hilbert-Schmidt. We also assume that Assumptions of Theorem  \ref{thm-BHR} are satisfied, i.e.  we assume   Assumption \ref{Assum:Usual hypotheses} and that
 $W=(W(t))_{t\geq 0}$ and  $\eta=(\eta(t))_{t\geq 0}$ are respectively $\ve$ and $\mathbb{R}$ valued Wiener processes defined on the filtered  probability space $(\Omega,
		\mathcal{F}, \mathbb{F}, \mathbb{P})$, and   $(\bu_0,\bn_0)\in \ve\times \sh^2$. Since we want to prove the existence of a local solution, we fix for the remainder of this section a finite time horizon $T>0$.
But contrary to the previous section, here
we do not   fix   	  $\eps\in (0,1]$  but instead consider  a family   $(\bue,\bne)_{\eps \in (0,1]}$, where   $(\bue,\bne): [0,\infty) \to \ve\times \sh^2$
the  unique strong solution to \eqref{GL-LC} guaranteed by Theorem  \ref{thm-BHR}.

It will convenient to introduce the following novation. Please note that we omit the  superscript ${}^\eps$.
\begin{align*}
&f_1(t)= \mathds{1}_{[0,\sigma^\eps]}(t) \left( -B(\bue(t), \bn^\eps(t))  - \Pi_L (\Div[\nabla \bne \odot \nabla \bne] )  \right),\; t\in [0,T]\\
&f_2(t)=\mathds{1}_{[0,\sigma^\eps]}(t),\; t\in [0,T],\\
&g_1(t)= \mathds{1}_{[0,\sigma^\eps]}(t)\left( -\bue(t) \cdot \nabla \bne(t) + \lvert \nabla \bne(t)\rvert^2 \bne(t) + \frac12 (\bne(t) \times \bh) \times \bh \right),\; t\in [0,T],\\
& g_2(t)= \mathds{1}_{[0,\sigma^\eps]}(t) \bne(t) \times \bh ,\; t\in [0,T].
\end{align*}
We then consider the following problem
\begin{subequations}\label{Eq:LinearExtbeyondrhokn}
	\begin{empheq}[left=\empheqlbrace]{align}
	&	d\bu(t) +\stok \bu(t) = f_1(t) dt + f_2 (t) dW,\text{    } t\in (0, T],\\
	&	d \bn (t) -\Delta \bn(t) dt = g_1(t) dt + g_2 (t) \times d\eta, \text{     }  t\in (0, T],\\
	&	\bu(0)=\bu_{0}) \text{ and } \bn(0) =\bn_{0}.
	\end{empheq}
\end{subequations}
This has a unique mild solution $(\bv^\eps,\bd^\eps)$ such that $(\bv^\eps,\bd^\eps)\in X_{[0,T]}= C([0, T]; \ve\times \sh^2  )\cap L^2(0,T; D(A)\times \sh^3 )$ almost surely.
Following the idea of \cite[page 128]{Brz+Masl+Seidler_2005},  we can prove that 
\begin{equation}\label{Eq:Equality}
(\bv^\eps(t\wedge \sigma^\eps), \bn^\eps (t\wedge \sigma^\eps) )= (\bue(t\wedge \sigma^\eps), \bne(t\wedge \sigma^\eps)), \; \forall t\ge0 \text{  $\mathbb{P}$-a.s.}
\end{equation}
Thanks to these observations and the uniform estimate in Lemma \ref{Energy estimates I} and Proposition \ref{Key uniform estimates} we obtain the following global estimates
\begin{prop}\label{Global Estimate}
	For any $p\in \mathbb{N}$, there exists $K_{11}(p)>0$ independent of $\eps \in (0,1]$ such that 
		\begin{align}
	&	\mathbb{E}\left( \sup_{t\in [0,T]}[\lvert A^\frac12 \bv^\eps(t)\rvert^{2 }_{\el^2}+ \lvert  \bd^\eps(t) \rvert^{2}_{\bH^2}  ]^p\right )\leq K_{11} ,\label{Eq:EstRegularityH1H2-EXT}\\
	&	\mathbb{E}\left[\int_0^{T} \left( \lvert A\bv^\eps \rvert^2_{\el^2}+ \lvert \bd^\eps \rvert^2_{\bh^3}  \right) ds \right]^p\leq K_{11}.	
	\label{Eq:EstRegularityH2H3-EXT}
	\end{align}
\end{prop}

Thanks to this proposition we can prove that the family  $(\bv^\eps, \bd^\eps)$ satisfies the following Aldous condition.
\begin{prop}\label{Aldous}
	There exists a constant $K_{12}>0$, independent of  $\eps\in (0,1)$, such that for every  $\kappa>0$ and every sequence $(\rho_n)_{n \in \mathbb{N}}$  of $ (0,T]$-valued stopping times,
	\begin{equation}
	\sup_{0\leq \theta \leq \kappa} \mathbb{E}\left(\lvert (\bv^\eps,\bd^\eps)((\rho_n +\theta)\wedge T)-  (\bv^\eps,\bd^\eps)(\rho_n )\rvert_{\sh \times \sh^1} \right) \leq K_{12}\kappa.
	\end{equation}
\end{prop}

Now, let us introduce the following notation
\begin{equation}\label{eqn-X_T}
\mathbf{X}_T= C([0,T]; \sh \times \sh^1)\cap C_{\text{weak}}([0,T]; \ve\times \sh^2)\cap L^2_{weak}(0,T; D(A) \times \bH^{3}).
\end{equation}
We also put
\begin{align*}
&\mathbf{X}^\alpha_T= C([0,T]; D(\stok^{\frac{\alpha-1}2 } \times \bH^{\alpha} )\cap L^2(0,T; D(\stok^\frac\alpha2 ))\times \bH^{1+\alpha}  ), \quad \alpha \in [1, 2),\\
&\mathbf{Y}_T = C([0,T]; \ve\times \mathbb{R}).
\end{align*}

The first corollary below follows from  Lemma \ref{Energy estimates I}.
\begin{cor}\label{cor-01}
We have 
\[(1-\lvert \bue \rvert^2) \to 0 \text{ in } L^2(\Omega; C([0,T);\el^2  )).\]
\end{cor}
The second corollary is  a consequence of Lemma \ref{Energy estimates I},    Propositions \ref{Key uniform estimates} and \ref{Aldous}, and \cite[Corollary 3.9]{ZB+EM}.

\begin{cor}\label{cor-Tight}
 The family  of laws of $[(\bv^\eps, \bd^\eps); (W,\eta); \sigma^\eps]$ is tight on $\mathbf{X}_T\times \mathbf{Y}_T\times [0,T]$.
\end{cor}

From  Corollaries  \ref{cor-01} and \ref{cor-Tight},  applying  the Jakubowski-Skorokhod representation theorem, \cite{Jakubowski} (see also  \cite[Theorem 3.11 ]{ZB+EM}),   we have the following result.

\begin{prop}
	There exist a new probability space $(\Omega, \mathcal{F}, \mathbb{P})$, not relabeled,   $\mathbf{X}_T\times \mathbf{Y}_T \times [0,T]$-valued sequence $(Z^\eps):=([(\tilde{\bv}^\eps, \tilde{\bd}^\eps); (W^\eps, \eta^\eps);  \tau^\eps])$ and  $\mathbf{X}_T\times \mathbf{Y}_T \times [0,T]$-valued random variable $Z:= [(\bv,\bd);(\tilde{W}, \tilde{\eta});\tau]$ such that
	\begin{align}
	& \text{law}_{\mathbf{X}_T\times \mathbf{Y}_T\times [0,T]} (Z^\eps)= \text{law}_{\mathbf{X}_T\times [0,T]} ([(\bv^\eps, \bd^\eps);(W,\eta);\sigma^\eps]),\\
	& Z^\eps \to Z \text{ in } (\mathbf{X}_T\cap \mathbf{X}^\alpha_T) \times \mathbf{Y}_T \times [0,T] \;\; \mathbb{P}\text{-a.s.},\label{Eq:Strong-Conv}\\
	& \mathbb{P}\text{-a.s.}  \mathds{1}_{[0,\tau^\eps]} \lvert \tilde{\bd}^\eps\rvert^2-1 \mathds{1}_{[0,\tau]}\to 0 \text{ in } L^q([0,T];\el^2  ) \forall q\in [2,\infty).		\label{Eq:Sphere-Conv}
	\end{align}
\end{prop}

In order to conclude the proof of Theorem \ref{thm-main} we need to pass to the limit. Thanks to the strong convergence \eqref{Eq:Strong-Conv} the passage to the equation for the velocity $\bv$ can be done as in \cite{ZB+EM}. The passage to the limit in the director equation needs special care. In particular, we need  the convergence \eqref{Eq:Sphere-Conv} and the following equivalence  result which can be established as in \cite{FA+AdeB+AH}.
	\begin{prop}
Let $\bu \in C([0,T]; \mathrm{H})\cap L^2(0,T; \ve)$ and consider the problems 
		\begin{equation}\label{Eq-UsualFrom of Eq Dir}
		d\bn+ \bu\cdot \nabla \bn= \Delta \bn + \lvert \nabla \bn \rvert^2 \bn + \frac12 G^2_\bh(\bn) + (\bn \times \bh)d\eta,  \;\; \lvert \bn \rvert=1
		\end{equation}
		and
		\begin{equation}\label{Eq-AltFrom of Eq Dir}
		\bn \times d\bn+ \bn \times (\bu\cdot \nabla \bn)= -\Div(\bn \times \nabla \bn  ) +\frac12 \bn \times G^2_\bh(\bn) + \bn \times (\bn \times \bh)d\eta,  \;\; \lvert \bn \rvert=1. 
		\end{equation}
		
If $\bn \in C([0,T]; \bH^1)\cap L^2(0,T; \bH^2)$ satisfies \eqref{Eq-UsualFrom of Eq Dir} then it satisfies \eqref{Eq-AltFrom of Eq Dir}, and vice versa. 
	\end{prop}
With this proposition at hand and \eqref{Eq:Sphere-Conv} we can now carry out as is done in \cite{FA+AdeB+AH} the passage to the limit in the equations for the director field $\bn$ and conclude the proof of Theorem \ref{thm-main}.

%
%


\begin{thebibliography}{99}
 	
 \bibitem{FA+AdeB+AH} F. Alouges,  A. de Bouard and A.  Hocquet,  \textit{A semi-discrete scheme for the stochastic Landau-Lifshitz equation.} { Stoch PDE: Anal Comp } \textbf{2}:281--315 (2014). 

 \bibitem{BBM_2012}  H. Bessaih,  Z. Brze{\'z}niak and A. Millet,   \textit{Splitting up method for the 2D stochastic Navier-Stokes equations},
\textit{Stoch. Partial Differ. Equ. Anal. Comput.} \textbf{2}(4): 433-470  (2014)

 	
 \bibitem{BD20} Z. Brze\'{z}niak and G. Dhariwal, G.   {Stochastic Constrained Navier--Stokes Equations on $\mathbb{T}^2$}, submitted,


 	\bibitem{BHP20} Z. Brze\'zniak, E. Hausenblas, P. Razafimandimby, \textit{ Strong solution to stochastic penalised nematic liquid crystals model driven by multiplicative Gaussian noise}. To appear in Indiana University Mathematics Journal, Preprint available on arXiv:2004.00590 (2020).
 	
 	\bibitem{BHP19} Z.~Brze\'zniak, E.~Hausenblas, and P. A.~Razafimandimby,  \textit{A note on the stochastic Ericksen-Leslie equations for nematic liquid crystals}, \newblock{  Discrete Contin. Dyn. Syst. Ser. B} \textbf{24} (11): 5785--5802 (2019).
 	
 	\bibitem{BHP18} Z.~Brze\'zniak, E.~Hausenblas, and P. A.~Razafimandimby,  \textit{Some results on the penalised nematic liquid crystals driven by multiplicative noise: weak solution and maximum principle}, \newblock{Stoch. Partial Differ. Equ. Anal. Comput.} \textbf{7}(3): 417--475 (2019).
 	

 	\bibitem{BHP13} Z.~Brze\'zniak, E.~Hausenblas, and P. A.~Razafimandimby,  \textit{Stochastic Nonparabolic dissipative systems modeling the flow of Liquid Crystals: Strong solution.}
 	\newblock{RIMS Symposium on Mathematical Analysis of Incompressible Flow, February 2013. RIMS K\^oky\^uroku } {1875}:41--73 (2014).
 	

 \bibitem{Brz+Masl+Seidler_2005} Z. Brze\'{z}niak, B. Maslowski and  J.Seidler,
\textit{Stochastic nonlinear beam equations,} {Probab. Theory Related Fields} \textbf{132}:119-149 (2005).


 \bibitem{ZB+EM} Z. Brze\'{z}niak and E. Motyl, \textit{Existence of a
martingale solution of Stochastic Naviour-Stokes equation in unbounded 2D and 3D domains},  J. Diff. Eq., \textbf{254}: 1627-1685 (2013).

 	\bibitem{BP_2001-Euler}   Z. Brze{\'z}niak and S Peszat, \textit{Stochastic two dimensional Euler Equations},
     Annals of Probability \textbf{ 29}:1796--1832 (2001).

 	\bibitem{YC+SK+YY-2018}Y. Chen, S. Kim and Y. Yu. \textit{Freedericksz transition in nematic liquid crystal flows in dimension two.}
 	\newblock{ SIAM J. Math. Anal.}  \textbf{50}(5):4838--4860 (2018).
 	
 	\bibitem{Climent} B. Climent-Ezquerra, F. Guill\'en-Gonz\'alez. A review of mathematical analysis of nematic and smectic-A liquid crystal models.
 	\newblock{\em European J. Appl. Math.}  \textbf{25}(1):133--153 (2014).
 	
 	
 	\bibitem{Gennes}{P.~G.~ de Gennes and J.~Prost.} {\it The Physics of Liquid Crystals}. Clarendon Press, Oxford
 	
 		\bibitem{Elw_1982} K.~D.~ Elworthy.
		{\it Stochastic Differential  Equations on  Manifolds}, London
		Math. Soc. LNS v 70, Cambridge University Press, 1982.
		


 	\bibitem{Ericksen} J.~L.~Ericksen. {\em Conservation laws for liquid
 	crystals.}
 	\newblock{ Trans. Soc. Rheology}, \textbf{5}: 23-34 (1961).
 	

 	
 	
 	\bibitem{Feng+Hong+Mei-2020}Z.~ Feng, M.-C.~Hong and Y.~Mei. \textit{Convergence of the Ginzburg-Landau approximation for the
 	Ericksen-Leslie system}. { SIAM J. Math. Anal.} \textbf{52}: 481-523 (2020)

 	

 	\bibitem{MH+JP-2018}{M.~Hieber and J. W.~Pr\"uss}.  \textit{Modeling and analysis of the Ericksen-Leslie equations for nematic liquid crystal flows.} In: Giga, Y., Novotny, A. (eds.) \newblock{ Handbook of Mathematical Analyis in Mechanics of Viscous Fluids}, pp. 1075--1134. Springer, Berlin, 2018.
 \bibitem{Hocquet2018} A. Hocquet,  \textit{Struwe-like solutions of the stochastic harmonic map flow}, Journal of Evolution Equations, \textbf{18}(3): 1189--1228  (2018).

 	
 	\bibitem{Hong} M.-C. Hong. \newblock{\em Global existence of solutions of the simplified Ericksen-Leslie system in dimension two.}
 	\newblock{ Calc. Var. Partial Differential Equations.} \textbf{40}:15--36 (2011).
 	
 	\bibitem{Hong12} M.-C. Hong and Z.~Xin. \textit{Global existence of solutions of the liquid crystal flow for the Oseen-Frank model in $\mathbb{R}^2$.}
 	\newblock{Adv. Math.}  \textbf{{231}}(3-4): 1364--1400 (2012).
 	
 	\bibitem{Hong14}M.-C. Hong, J. Li and Z. Xin. \textit{Blow-up criteria of strong solutions to the Ericksen-Leslie system in $\mathbb{R}^3$.}
 	\newblock{ Comm. Partial Differential Equations.}  \textbf{39}(7):1284--1328 (2014).
 	
 	\bibitem{MCH+YM-2019} M.-C.. Hong and Y. Mei . \textit{Well-posedness of the Ericksen-Leslie system with the Oseen-Frank energy in $L^3_{\text{uloc}}(\mathbb{R}^3)$.}
 	\newblock{ Calc. Var. Partial Differential Equations}  \textbf{58}(1): Art. 3, 38 (2019).
 	
 	
 	\bibitem{Huang14}J. Huang, F. Lin, Fanghua and C. Wang. \textit{Regularity and existence of global solutions to the Ericksen-Leslie system in $\mathbb{R}^2$.}
 	\newblock{\em Comm. Math. Phys.}  \textbf{331}(2): 805--850 (2014).
 	
 	\bibitem{Huang16} T.~Huang, F. Lin, C. Liu and C. Wang. \textit{Finite time singularity of the nematic liquid crystal flow in dimension three.}
 	{ Arch. Ration. Mech. Anal.}  \textbf{221}(3): 1223--1254 (2016).

\bibitem{Jakubowski}A. Jakubowski, \textit{The almost sure Skorokhod representation for subsequences in nonmetric spaces}, { Teor. Veroyatnost. i Primenen.} \textbf{42}(1):209-216 (1997); translation in {Theory Probab. Appl.} \textbf{42}(1) :167-174 (1998).
 	
 		\bibitem{Kar-Shr-96} I. Karatzas and S. Shreve. \textit{ Brownian Motion
			and Stochastic Calculus}, 2nd edition, Graduate Texts in
		Mathematics {\bf 113},  Springer Verlag, Berlin Heidelberg New
		York 1996.
\bibitem{Kunita-90} H. Kunita. \textit{Stochastic flows and stochastic differential equations}, Cambridge University Press, 1990.

 	\bibitem{Leslie} F.~M.~Leslie. \textit{Some constitutive equations for liquid crystals}. \newblock{Arch. Rational Mech. Anal.} \textbf{28}(4):265-283  (1968).
 	
 	\bibitem{JL+ET+ZX-2016} J. Li, E. S. Titi and Z. Xin. \textit{On the uniqueness of weak solutions to the Ericksen-Leslie liquid crystal model in $\mathbb{R}^2$.}
 	\newblock{ Math. Models Methods Appl. Sci.}  \textbf{26}(4): 803--822 (2016).
 	
 	\bibitem{Lin+Wang16} F. Lin and C. Wang. \textit{Global existence of weak solutions of the nematic liquid crystal flow in dimension three.}
 	\newblock{ Comm. Pure Appl. Math.}  \textbf{69}(8):1532--1571 (2016).
 	
 	\bibitem{LLW} F. Lin, J. Lin and C. Wang. \newblock{{\em Liquid crystal flows in two dimensions.}}  \newblock{ Arch. Rational Mech. Anal.} \textbf{197}:297-336 (2010).
 	\bibitem{Lin-Liu} F. H. Lin and C. Liu. \newblock{\em Nonparabolic dissipative systems modeling the flow of Liquid Crystals.}
 	\newblock{ Communications on Pure and Applied Mathematics}, Vol. \textbf{XLVIII}:501--537 (1995).
 	%
 	\bibitem{Lin-Liu2} F.-H. Lin and C. Liu. \newblock{\em Existence of solutions for the Ericksen-Leslie System},
 	\newblock{ Arch. Rational Mech. Anal.} \textbf{154}:135-156 (2000).
 	
 	\bibitem{Lin+Wang-2014} F. Lin and C. Wang,  \textit{Recent developments of analysis for hydrodynamic flow of nematic liquid crystals}, \newblock{ Philos. Trans. R. Soc. Lond. Ser. A Math. Phys. Eng. Sci.} \textbf{372}: 20130361, 18 pp (2014).
 	
 	\bibitem{Lin-Wang} F. Lin and C. Wang, \newblock{\em On the uniqueness of heat flow of harmonic maps and hydrodynamic flow of nematic liquid crystals.}
 	\newblock{\em Chinese Annals of Mathematics, Series B.} \textbf{31B}(6): 921-938 (2010).
 	
 	
 	\bibitem{Medjo1} T.~ Tachim Medjo,  \textit{On the existence and uniqueness of solution to a stochastic simplified liquid crystal model.} \newblock{Communications on Pure \& Applied Analysis.}\textbf{18}(5):2243--2264 (2019),.
 	
 		\bibitem{Metivier_1982}
		M.~M\'etivier,  {\it Semimartingales. A course on stochastic processes.} Vol.~\textbf{2} of de
	 Gruyter Studies in Mathematics, 1982.


		\bibitem{Pardoux_1979} E.~Pardoux, \textit{ Stochastic partial differential equations and
			filtering of diffusion processes},  {Stochastics} \textbf{3}: 127-167 (1979)

 	
 	\bibitem{Stewart} I. W. Stewart, \newblock{\em The Static and Dynamic Continuum Theory of Liquid Crystals: A Mathematical Introduction}, CRC Press, Boca Raton, FL, 2004.
 	\bibitem{Struwe_1985} M. Struwe. \textit{On the evolution of harmonic maps of Riemannian surfaces}. \newblock{ Commun. Math. Helv.} \textbf{60}: 558--581 (1985).
 	\bibitem{Temam95} R. Temam,  \textit{Navier--Stokes equations and nonlinear functional analysis,} 2nd ed. SIAM, Philadelphia, PA.  1995. 
 	
 	\bibitem{VF88}  M. J. Vishik  and  A.V. Fursikov  \textit{Mathematical problems of statistical hydromechanics.} Kluwer Academic Publishers, Dordrecht, 1988. 
 	
 	
 	
 	
 	
 		\bibitem{Walkington} N. J. Walkington, \textit{Numerical approximation of nematic liquid crystal flows governed by the Ericksen-Leslie equations}, \newblock{ESAIM: Mathematical Modelling and Numerical Analysis} \textbf{45}:523--540 (2011).
 	
 	
 	
 	\bibitem{Wang2014}M.~Wang and W.~Wang. \newblock{\em Global existence of weak solution for the 2-D Ericksen-Leslie
 		system.}
 	\newblock{ Calc. Var. Partial Differential Equations,}  \textbf{51}(3-4): 915--962 (2014).
 	
 	\bibitem{WZZ}W.~Wang, P.~Zhang  and Z.~Zhang. \newblock{\em Well-posedness of the Ericksen-Leslie system.}
 	\newblock{ Arch. Ration. Mech. Anal.}  \textbf{210}(3): 837--855 (2013).
 	
 	\bibitem{Wang2016}M.~Wang, W.~Wang and Z.~Zhang. {\em On the uniqueness of weak solution for the 2-D Ericksen-Leslie system.}
 	\newblock{ Discrete Contin. Dyn. Syst. Ser. B}  \textbf{21}(3):919--941 (2016).
 	
 	
 	
 	
 	\bibitem{WZZ15}W.~Wang, P.~Zhang  and Z.~Zhang. \textit{The small Deborah number limit of the Doi-Onsager equation to the Ericksen-Leslie equation.}
 	\newblock{ Comm. Pure Appl. Math.}  \textbf{68}(8): 1326--1398 (2015).
 	
 	



 	
 	 \end{thebibliography}
\end{document}